\documentclass[12pt,a4paper,reqno]{amsart}

\usepackage{amsfonts}
\usepackage{amsmath}
\usepackage{amssymb}
\usepackage{amsthm}
\usepackage{amsxtra}
\usepackage{bbm}
\usepackage{braket}
\usepackage{comment}
\usepackage{extarrows}

\usepackage{graphicx}
\usepackage{latexsym}

\usepackage{mathrsfs}
\usepackage{yhmath}

\usepackage{nicematrix}
\usepackage{mathtools}

\usepackage{float}
\usepackage{booktabs}
\usepackage{verbatim}
\usepackage{xcolor}

\usepackage[pdfborder={0 0 0}]{hyperref} 
\hypersetup{breaklinks, raiselinks}

\usepackage{bm}

\usepackage{cleveref}

\usepackage{a4wide}
\usepackage{fullpage}

\theoremstyle{plain}
\newtheorem{Theorem}{Theorem}[section]

\newtheorem{Lemma}[Theorem]{Lemma}

\theoremstyle{definition}

\newtheorem{Assumptions and Discussion}[Theorem]{Assumptions and Discussion}

\theoremstyle{remark}

\newtheorem*{acknowledgment*}{Acknowledgment}

\numberwithin{equation}{section}

\usepackage[shortlabels]{enumitem}
\SetEnumerateShortLabel{a}{\textup{(\alph*)}}
\SetEnumerateShortLabel{A}{\textup{(\Alph*)}}
\SetEnumerateShortLabel{1}{\textup{(\arabic*)}}
\SetEnumerateShortLabel{i}{\textup{(\roman*)}}
\SetEnumerateShortLabel{I}{\textup{(\Roman*)}}

\def\tr{\operatorname{tr}}

\begin{document}

\title{Bernstein-type theorem for stationary hypersurfaces of the Euler-Dierkes-Huisken functional}

\author[Hongbin Cui, Jiahuan Li, Xiaowei Xu]{Hongbin Cui \quad Jiahuan Li\quad  Xiaowei Xu}

\address{Hongbin Cui, School of Mathematical Sciences, University of Science and Technology of China, Hefei, Anhui, 230026, P.R.~China}
\email{cuihongbin@ustc.edu.cn}

\address{Jiahuan Li, School of Mathematical Sciences, University of Science and Technology of China, Hefei, Anhui, 230026, P.R.~China}
\email{jiahuan@mail.ustc.edu.cn}

\address{Xiaowei Xu (The corresponding author.), School of Mathematical Sciences, University of Science and Technology of China, Hefei, Anhui, 230026, P.R.~China}
\email{xwxu09@ustc.edu.cn}

\begin{abstract}
We say that a hypersurface $\Sigma \subset\mathbb{R}^{n+1}$ is $\alpha$-stationary if it is a critical point of the Euler-Dierkes-Huisken functional $\mathcal{E}_\alpha(\Sigma)=\int_\Sigma|X|^\alpha\, d\mathcal{H}^n$, introduced by Dierkes and Huisken in \cite{[DH-24]}. In this paper, we prove that every smooth, complete, connected, embedded $\alpha$-stationary hypersurface in $\mathbb{R}^{n+1}$ passing through the origin with $\alpha>0$ is a linear hyperplane.

\end{abstract}
\maketitle

\section{Introduction}
In 1744, L. Euler considered the least moment of inertia for planar curves with fixed endpoints and introduced the functional
\begin{equation*}
\mathcal{E}_2(X(t))=\int_a^b|X'(t)|^2\, ds,
\end{equation*}
where $X(t)=(x_1(t),x_2(t))\subset\mathbb{R}^2$, $t\in[a,b]$, and $ds=|X'(t)|dt$. M. Mason (\cite{[M-1906]}), L. Tonelli (\cite{[T-1921]}), C. Carathedory (\cite{[C-1935]}) studied regular extremals of a more general functional
\begin{equation*}
\mathcal{E}_\alpha(X(t))=\int_a^b|X'(t)|^\alpha\, ds,
\end{equation*}
where $\alpha>0$. Recently, U. Dierkes and G. Huisken (\cite{[DH-24]}) considered the $n$-dimensional analogue functional
\begin{equation}\label{A-1}
\mathcal{E}_\alpha(\Sigma):=\int_\Sigma|X|^\alpha\, d\mathcal{H}^n
\end{equation}
for hypersurfaces in $\mathbb{R}^{n+1}$, where $X$ is the position of a point in $\mathbb{R}^{n+1}$, $\alpha\in \mathbb{R}$, and $\mathcal{H}^n$ is the $n$-diemnsional Hausdorff measure on $\mathbb{R}^{n+1}$. We call (\ref{A-1}) the \emph{Euler Dierkes-Huisken functional} (EDH-functional) in the sequel. The Euler-Lagrange equation of EDH-functional is
\begin{equation}\label{A-2}
H=\alpha\, |X|^{-2}\, \langle X,\nu\rangle,
\end{equation}
where $H$ is the mean curvature of $\Sigma$, and $\nu$ is the unit normal vector field of $\Sigma$.
A hypersurface is said to be \emph{$\alpha$-stationary} or \emph{EDH-stationary} if it satisfies (\ref{A-2}) in $\mathbb{R}^{n+1}\setminus \{0\}$. We say that an $\alpha$-stationary hypersurface $\Sigma$ \emph{passes through the origin} $0\in \mathbb{R}^{n+1}$ if $0\in \Sigma$ and $\Sigma\setminus \{0\}$ satisfies (\ref{A-2}). In particular, if $\Sigma=\{(x,u(x))\}$ is a graph of a smooth function $u$ over 
$\mathbb{R}^n$, then (\ref{A-2}) becomes
\begin{equation}\label{A-3}
\mbox{div}\big(\frac{\nabla u}{\sqrt{1+|\nabla u|^2}}\big)=\alpha\,\frac{u-\langle x,\nabla u\rangle}{(|x|^2+u^2)\sqrt{1+|\nabla u|^2}},
\end{equation}
where $\nabla u$ denotes the gradient of $u$ in $\mathbb{R}^n$. Note that if $\alpha=0$, the EDH-functional reduces to the area functional, and (\ref{A-3}) is precisely the minimal surface equation. 

After S. Bernstein (\cite{[B-1927]}) proved that a minimal graph over $\mathbb{R}^2$ must be a plane, it is natural to ask whether a minimal graph over $\mathbb{R}^n$ is necessarily a hyperplane for $n\geq 3$. This question became known as the \emph{Bernstein problem}. The work of Fleming (\cite{[F-1962]}), De Giorgi (\cite{[D-1965]}), Almgren (\cite{[A-1966]}), and J. Simons (\cite{[S-1968]}) showed that the Bernstein problem holds for $n\leq 7$. However, E. Bombieri, E. De Giorgi, and E. Giusti (\cite{[BDG-1969]}) constructed counterexamples for $n\geq 8$.

It is known that a minimal graph is a particular class of stable minimal hypersurface. Naturally, this leads to the \emph{stable Bernstein problem}: must a complete, stable minimal hypersurface in $\mathbb{R}^{n+1}$ be a hyperplane? Significant progress has been made on this problem. Do Carmo and C.K. Peng (\cite{[DP-1979]}), D. Fischer-Colbrie and R. Schoen (\cite{[FS-1980]}), and A.V. Porogelov (\cite{[P-1981]}) gave a positive answer for $n=2$. R. Schoen, L. Simon and S.T. Yau (\cite{[SSY-1975]}), R. Schoen and L.Simon (\cite{[SS-1981]}) studied the problem under the Euclidean volume assumptions for $n\leq 6$. C. Bellettini (\cite{[B-2025]}) developed their approach via De Giorgi's iteration. Recently, O. Chodosh and Ch. Li \cite{[CL-2024]} gave a positive answer for $n=3$. Shortly after, two alternative proofs were provided independently by G. Catino, P. Mastrolia and A. Roncoroni \cite{[CMR-2024]}, and by O. Chodosh and Ch. Li (\cite{[CL-2023]}). Subsequently, O. Chodosh, Ch. Li, P. Minter and D. Stryker (\cite{[CLMS-2024]}) applied the strategy in \cite{[CL-2023]} to solve the problem for $n=4$. L. Mazet (\cite{[M-2024]}) settled the case $n=5$. The problem remains open for $n=6$.

The Bernstein problem for hypersurfaces with respect to weighted area functionals has also drawn considerable attention in geometry and analysis. U. Dierkes (\cite{[D-1993]}, \cite{[D-1995]}) established Bernstein-type theorems for stable graphs and area-minimizing hypersurfaces associated with the weighted area functional, where the weight is a power of the last coordinate. L. Wang (\cite{[W-2011]}), Q. Guang and J. Zhu (\cite{[GZ-2017]}) obtained  Bernstein theorems for graphic self-shrinkers. C. Mooney and Y. Yang (\cite{[MY-2024]}) have completed the solution to the anisotropic Bernstein problem, in which the weight function depends on the normal vector. For results on Bernstein-type theorems for other nonlinear geometric PDEs, we refer the reader to \cite{[M-2024]}.

In \cite{[DH-24]}, U. Dierkes and G. Huisken established fundamental properties of $\alpha$-stationary hypersurfaces and investigated the stability and area-minimizing properties​ of hypercones associated with the functional (\ref{A-1}). H.B. Cui and X.W. Xu (\cite{[CX-2025]}) extended the theory of Dierkes and Huisken for​ higher codimension surfaces in Euclidean space. U. Dierkes and  R. López (\cite{[DL-2026]}) proved that an axially symmetric $\alpha$-stationary surface in $\mathbb{R}^3\setminus \{0\}$  with $\alpha>0$ is an entire graph defined on the plane perpendicular to the axis of symmetry. For the classification of $\alpha$-stationary surfaces in $\mathbb{R}^3$ under various geometric assumptions, we refer the reader to \cite{[DL-2026]}, \cite{[L-2025-1]}, and \cite{[L-2025-2]}. In this paper, we prove the following Bernstein-type theorem for $\alpha$-stationary hypersurfaces.

\begin{Theorem}\label{Theorem 1.1}
Let $\Sigma\subset \mathbb{R}^{n+1}$ be a smooth, complete, connected, embedded $\alpha$-stationary hypersurface passing through the origin with $\alpha>0$, then $\Sigma$ is a linear hyperplane.
\end{Theorem}

\emph{Remark}. Since any hyperplane that is $\alpha$-stationary must contain the origin, the assumption that $\Sigma$ passes through the origin is natural. 
The smoothness assumption on $\Sigma$ at the origin is essential.  The smoothness assumption of $\Sigma$ at the origin is essential. Indeed, the graph $\{(x, m|x|)\,|\,x\in\mathbb{R}\}\subset \mathbb{R}^2$ is $\alpha$-stationary for $m\neq 0$ and $\alpha>0$  in $\mathbb{R}^2$, but it is not smooth at the origin. Furthermore, the Simon's cones $\{(x,y)\,|\, x,y\in\mathbb{R}^{\ell+1},\;|x|=|y|\}\subset \mathbb{R}^{\ell+1}\times \mathbb{R}^{\ell+1}$ provide examples of $\alpha$-stationary (\cite{[DH-24]}) hypersurfaces for any $\alpha>0$, although they are singular at the origin. It is worth noting that neither the stability nor the area-minimizing property is required in the proof of our Bernstein-type theorem.

This paper is organized as follows. In Section 2, we show that the Taylor coefficients of an $\alpha$-stationary graph defined by a smooth function $u$ vanishes at the origin. Section 3 is dedicated to establishing the local flatness of such graphs. Finally, in Section 4, we complete the proof of the main theorem.

\section{Local estimates for stationary graph}
In this section, we prove that a smooth function $u$ whose graph is $\alpha$-stationary has vanishing Taylor coefficients at the origin, i.e., all its partial derivatives of every order​ vanish at 0.

\begin{Theorem}\label{Theorem 2.1}
Let $u\in C^\infty(B_{r_0})$ satisfies $u(0)=0$ and $\nabla u(0)=0$. If $u$ satisfies the equation \emph{(\ref{A-3})} with $\alpha >0$ on $B_{r_0}\setminus\{0\}$, then $u$ has vanishing Taylor coefficients at the origin. Moreover, for a fixed multi-index $\beta\in\mathbb{N}^n$, we have
\begin{equation*}
|\nabla^\beta u|\leq C_{\beta,m}\,|x|^m,
\end{equation*}
for all $m\in \mathbb{N}$, where $C_{\beta,m}>0$ is a constant depends only on $\beta$ and $m$.
\end{Theorem}

\emph{Proof}. We rewrite equation (\ref{A-3}) as
\begin{equation}\label{B-1}
\big(\delta_{ij}-\frac{u_i u_j}{1+|\nabla u|^2}\big)u_{ij}=\alpha\,\frac{u-\langle x,\nabla u\rangle}{|x|^2+u^2}
\end{equation}
in $B_{r_0}\setminus \{0\}$. We will prove that equation (\ref{B-1}) leads to a contradiction if $u$ has a finite order term in its Taylor expansion at the origin.

Since $u(0)=0$ and $\nabla u(0)=0$, the Taylor expansion of $u$ at the origin takes the form 
\begin{equation}\label{B-2}
u(x)=P_k(x)+o(|x|^k),
\end{equation}
where $k\geq 2$ and $P_k(x)$ is a non-trivial homogeneous polynomial of degree $k$.
This implies
\begin{equation}\label{B-3}
\nabla u(x)=\nabla P_k(x)+o(|x|^{k-1}),\hspace{0.5cm} \nabla^2 u(x)=\nabla^2 P_k(x)+o(|x|^{k-2}).
\end{equation}
Since $k\geq 2$, it follows from (\ref{B-3}) that
\begin{equation*}
\frac{u_i u_j}{1+|\nabla u|^2}\,u_{ij}=o(|x|^{k-2}).
\end{equation*}
Thus, the left hand side of (\ref{B-1}) is equal to 
\begin{equation}\label{B-4}
\Delta P_k(x)+o(|x|^{k-2}).
\end{equation}
For the right hand side, Euler's identity for homogeneous polynomials gives 
\begin{equation*}
\langle x, \nabla P_k(x)\rangle = kP_k(x),
\end{equation*}
and hence 
\begin{equation*}
u-\langle x, \nabla u\rangle=(1-k)P_k(x)+o(|x|^k).
\end{equation*}
Using $k\geq 2$ once more, we obtian 
\begin{equation*}
\frac{1}{|x|^2+u^2}=|x|^{-2}+o(|x|^{-2}).
\end{equation*}
Thus the right hand side of (\ref{B-1}) is equal to
\begin{equation}\label{B-5}
\alpha(1-k)\frac{P_k(x)}{|x|^2}+o(|x|^{k-2}).
\end{equation}
Substituting (\ref{B-4}) and (\ref{B-5}) into (\ref{B-1}), and comparing the homogeneous terms of degree $k-2$, we obtain
\begin{equation}\label{B-6}
|x|^2\Delta P_k(x)=\alpha (1-k) P_k(x).
\end{equation}

We first treat the case $n=1$. Writing $P_k(x)=cx^k$ with $c\neq 0$ and $k\geq 2$, we have $|x|^2\Delta P_k(x)=k(k-1)P_k(x)$. Then (\ref{B-6}) gives $k(k-1)=\alpha (1-k)$, which contradicts $k\geq 2$ and $\alpha>0$.

We next treat the case $n\geq 2$. By the standard harmonic decomposition of homogeneous polynomials (see Chapter 5 in \cite{[ABR-01]}), $P_k(x)$ can be written uniquely as
\begin{equation}\label{B-7}
P_k(x)=\sum\limits_{i=0}^{[k/2]}|x|^{2i}H_{k-2i}(x),
\end{equation}
where $H_{k-2i}(x)$ is a homogeneous harmonic polynomial of degree $k-2i$. Since $\Delta H_{k-2i}=0$ and 
$\langle x, \nabla H_{k-2i}(x)\rangle=(k-2i)H_{k-2i}$, we compute
\begin{equation*}
\Delta(|x|^{2i}H_{k-2i})=2i(2i+n-2)|x|^{2i-2}H_{k-2i}+4i(k-2i)|x|^{2i-2}H_{k-2i},
\end{equation*}
which yields
\begin{equation}\label{B-8}
|x|^2\Delta(|x|^{2i}H_{k-2i})=2i(2k-2i+n-2)|x|^{2i}H_{k-2i},
\end{equation}
for all $i\in\{0,\ldots, [k/2]\}$. Substituting (\ref{B-7}) and (\ref{B-8}) into (\ref{B-6}), and using uniqueness of harmonic decomposition, we find that the coefficients must satisfy
\begin{equation*}
2i(2k-2i+n-2)=\alpha(1-k)
\end{equation*}
for some $i\in\{0,\ldots, [k/2]\}$, which contradicts $\alpha>0$ and $k\geq 2$. Thus no such solution exists, and the theorem follows from the Taylor expansion of a smooth function.

\hfill$\Box$

\emph{Remark}. Theorem \ref{Theorem 2.1} does not hold​ for a​ minimal graph, i.e., the case $\alpha=0$. However, it holds for some other $\alpha$, such as negative irrationals. 
On the other hand, from Theorem \ref{Theorem 2.1}, the coefficients $a^{ij}(x)$
in (\ref{B-1}) can be continuously extended​ to the origin by setting​ $a^{ij}(0)=\delta_{ij}$, for smooth functions have vanishing Taylor coefficients.
	​
\section{Local flatness for stationary graph}
In this section, we prove the local flatness of smooth $\alpha$-stationary graphs through the origin, using the strong unique continuation result of Ch.L. Lin, G. Nakamura and J.N. Wang in \cite{[LNW-11]}.

We rewrite (\ref{B-1}) as
\begin{equation}\label{C-1}
a^{ij}(x)u_{ij}+b^i(x)u_i+c(x)u=0,
\end{equation}
where
\begin{equation*}
b^{ij}(x):=\frac{u_i u_j}{1+|\nabla u|^2},\hspace{0.3cm}
a^{ij}(x):=\delta_{ij}-b^{ij}(x), \hspace{0.3cm}
b^i(x):=\frac{\alpha x_i}{|x|^2+u^2},\hspace{0.3cm}
 c(x):=-\frac{\alpha}{|x|^2+u^2}.
\end{equation*}
To apply the theory in \cite{[LNW-11]}, we define the transformation as
\begin{equation}\label{C-2}
w(x):=|x|^su(x),
\end{equation}
where $s=\alpha/2$. Clearly, $w(x)$ also has vanishing Taylor coefficients at the origin if $u$ does.
Away from the origin, a direct computation shows that
\begin{equation}\label{C-3}
u_i=|x|^{-s}w_i-s|x|^{-s-2}x_i w,
\end{equation}
and
\begin{equation}\label{C-4}
u_{ij}=|x|^{-s}w_{ij}-s|x|^{-s-2}(x_iw_j+x_jw_i+\delta_{ij}w)+s(s+2)|x|^{-s-4}x_ix_jw.
\end{equation}
Substituting (\ref{C-3}) and (\ref{C-4}) into (\ref{C-1}), and multiplying by $|x|^s$ on both sides, we obtain
\begin{equation}\label{C-5}
a^{ij}(x)w_{ij}+\tilde{b}^i(x)w_i+\tilde{c}(x)w=0,
\end{equation}
where
\begin{equation}\label{C-6}
\tilde{b}^i(x)=b^i(x)-2s|x|^{-2}a^{ij}(x)x_j,
\end{equation}
and
\begin{equation}\label{C-7}
\tilde{c}=-s|x|^{-2}a^{ij}(x)\delta_{ij}+s(s+2)|x|^{-4}a^{ij}(x)x_ix_j-s|x|^{-2}b^i(x)x_i+c(x).
\end{equation}

\begin{Lemma}\label{Lemma 3.1}
Let $P(x,D)=a^{ij}(x)\,\partial_{ij}$ be the second-order differential operator in \emph{(\ref{C-1})}. For any smooth
function $u(x)$ with vanishing Taylor coefficients at the origin that satisfies \emph{(\ref{C-1})}, then $w(x):=|x|^su(x)$ satisfies
\begin{equation*}
|P(x,D)w|\leq C_1\,|x|^{-2}|w|+C_2\, |x|^{-1}|\nabla w|,
\end{equation*}
where $C_2\rightarrow 0$ as $|x|\rightarrow 0$.
\end{Lemma}

\emph{Proof}. Substituting $a^{ij}(x)=\delta_{ij}-b^{ij}(x)$ into (\ref{C-6}) and (\ref{C-7}), using the definition of $b^i(x)$ and $c(x)$, we obtain
\begin{eqnarray}\label{C-8}
\tilde{b}^i(x)&=&\frac{2b^{ij}x_j}{|x|^2}-\frac{\alpha u^2x_i}{|x|^2(|x|^2+u^2)}\nonumber\\
&\leq& \Big(2 |B|+\frac{\alpha u^2}{|x|^2+u^2}\Big)\,\frac{1}{|x|},
\end{eqnarray}
where $|B|$ is the norm of matrix $B:=(b^{ij}(x))_{n\times n}$, and 
\begin{eqnarray}\label{C-9}
\tilde{c}(x)&=&\frac{s(s-n+2)}{|x|^2}-\frac{2s(s+1)}{|x|^2+u^2}+\frac{sb^{ii}(x)}{|x|^2}-\frac{s(s+2)b^{ij}(x)x_ix_j}{|x|^4}\nonumber\\
&\leq&\frac{C(s)}{|x|^2}+\Big(s|\tr(B)|+\frac{s(s+2)|b^{ij}(x)x_ix_j|}{|x|^2}\Big)\,\frac{1}{|x|^2},
\end{eqnarray}
where $C(s)=s(n+3s+4)$ depends only on $\alpha$. Since $u(x)$ has vanishing Taylor coefficients at the origin, (\ref{C-8}) and (\ref{C-9}) imply
\begin{equation}\label{C-10}
|\tilde{b}^i(x)|\leq \epsilon(|x|)\,|x|^{-1},\hspace{0.3cm} |\tilde{c}(x)|\leq (C(s)+\epsilon(|x|))|x|^{-2},
\end{equation}
where $\epsilon(|x|)\rightarrow 0$ as $|x|\rightarrow 0$. The assertion now follows from (\ref{C-5})and (\ref{C-10}).

\hfill$\Box$

To obtain local flatness, we appeal to Theorem 1.2 from \cite{[LNW-11]}.

\begin{Theorem}\label{Theorem 3.2} \emph{(\cite{[LNW-11]})}
Let $\Omega\subset \mathbb{R}^n$ be a connected open set containing $0$, where $n\geq 2$. Let $P(x,D)=a^{ij}(x)\partial_{ij}$ be an elliptic differential operator on $\Omega$ with $a^{ij}(0)$ real symmetric matrix and $a^{ij}(x)$ Lipschitz on $\Omega$. If $w(x)\in H^1_{loc}(\Omega)$ is a nontrivial solution satisfying
\begin{equation*}
|P(x,D)w|\leq C_1\,|x|^{-2}|w|+C_2\, |x|^{-1}|\nabla w|
\end{equation*}
with $C_2$ is sufficiently small, then there exist constants $C_0>0$, depending only on $n$ and $w(x)$, and a constant $m_0$ depending only on $P(x,D)$ and $w(x)$, such that
\begin{equation*}
\int_{|x|<r}|w(x)|^2dx\geq C_0 r^{m_0},
\end{equation*}
for sufficiently small $r$.
\end{Theorem}

\emph{Proof}. A detailed proof can be found in Section 3 of \cite{[LNW-11]}.

\hfill$\Box$

\begin{Theorem}\label{Theorem 3.3}
Let $u\in C^\infty(B_{r})$ satisfies $u(0)=0$ and $\nabla u(0)=0$. If $u$ solves \emph{(\ref{A-3})} with $\alpha >0$ in $B_{r}\setminus\{0\}$, then $u\equiv 0$ on $B_{r_0}$ for sufficiently small $r_0$.
\end{Theorem}

\emph{Proof}. We first treat the case $n=1$. Set $u(x):=xw(x)$ for $x>0$. Clearly, $u(x)$ has vanishing Taylor coefficients at the origin, and so do $w(x)$ and $w'(x)$. Then the equation (\ref{B-1}) becomes
\begin{equation}\label{C-11}
xw''(x)+a(x,w,w')w'(x)=0,
\end{equation}
where $a(x,w,w')=2+\alpha[1+(w+xw')^2]/(1+w^2)>0$ is smooth for $x>0$. Equation (\ref{C-11}) implies 
\begin{equation}\label{C-12}
w'(x)=w'(x_0)e^{-\int_{x_0}^x\frac{a}{t}\,dt},
\end{equation}
for any $x_0>0$.
Taking the limit $x\rightarrow 0^+$ in (\ref{C-12}) yields $w'(0)=w'(x_0)e^{\int_0^{x_0} \frac{a}{t}\,dt}$, which contradicts $w'(0)=0$ unless $w'(x_0)=0$. Since $x_0>0$ is arbitrary and $w'(0)=0$, it follows that $w'(x)\equiv0$. Consequently, $w(x)\equiv 0$ in light of $w(0)=0$, and thus $u(x)\equiv 0$.

We next treat the case $n\geq 2$. It follows from Theorem \ref{Theorem 2.1} that $u(x)$ has vanishing Taylor coefficients at the origin. Clearly, $w(x)=|x|^s u(x)$ with $s=\alpha/2$ shares the same property. That is, there is a constant $C_m>0$ such that 
\begin{equation}\label{C-13}
|w(x)|^2\leq C_m\,|x|^m,
\end{equation}
for all $m\in\mathbb{N}$.
We now extend the coefficients $a^{ij}(x)$ to $0$ by setting $a^{ij}(0)=\delta_{ij}$. The extended coefficients remain Lipschitz in a small ball. Then, Lemma \ref{Lemma 3.1} ensures​ that $w(x)$ satifies the hypotheses of Theorem \ref{Theorem 3.2}. We claim that $u\equiv 0$ in sufficiently small ball. Suppose, to the contrary, that $w(x)$ is not identically zero. Then Theorem \ref{Theorem 3.2} yields
\begin{equation*}
\int_{B_r}|w(x)|^2 dx \geq C_0 |x|^{m_0},
\end{equation*}
for all sufficiently small $r>0$, which contradicts inequality (\ref{C-13}). Thus, $u(x)\equiv 0$ in a sufficiently small neighborhood of the origin.

\hfill$\Box$

\section{Proof of the main theorem}
We complete the proof of the main theorem in this section.

\begin{Theorem}\label{Theorem 4.1}
Let $\Sigma\subset \mathbb{R}^{n+1}$ be a smooth, connected, embedded, complete $\alpha$-stationary hypersurface passing through the origin with $\alpha>0$, then $\Sigma$ is a linear hyperplane.
\end{Theorem}

\emph{Proof}. Since $\Sigma$ is embedded and $0\in\Sigma$, there exists a neighborhood of $0$ in $\Sigma$ that can be written as a graph over $T_0\Sigma$. After rotating $\mathbb{R}^{n+1}$ if necessary, we identify $T_0\Sigma$ with $\mathbb{R}^n=\{(x,0)\}\subset \mathbb{R}^{n+1}$. Consequently, $\Sigma$ can be represented as the graph of a smooth function $u(x)$ in $B_r(0)\subset \mathbb{R}^n$ with $u(0)=0$ and $\nabla u(0)=0$, where $u(x)$ solves (\ref{B-1}) in $B_{r}(0)\setminus \{0\}$. By Theorem \ref{Theorem 2.1} and Theorem \ref{Theorem 3.3}, we conclude that $u(x)\equiv 0$ in some $B_{r_0}(0)$. Hence, $\Sigma$ coincides with the hyperplane $\mathbb{R}^n$ in a neighborhood of $0$.

Define the set
\begin{equation*}
\mathcal{O}:=\{p\in\Sigma\,|\, \mbox{there exists a neighborhood} \;\mathcal{O}_p\subset \Sigma \;\mbox{of}\;p\;\mbox{such that}\; \mathcal{O}_p\subset \mathbb{R}^n\}.
\end{equation*}
Clearly, $\mathcal{O}$ is non-empty and open in $\Sigma$. We claim that $\mathcal{O}$ is also closed in $\Sigma$. Let $\{p_i\}\subset \mathcal{O}$ be a sequence coverging to $ p\in \Sigma\setminus \{0\}$. Since $p_i\rightarrow p$ and $p_i\in \mathbb{R}^n$, it follows that $p\in\mathbb{R}^n$. Moreover, by the smoothness of $\Sigma$, we have $T_p\Sigma=\mathbb{R}^n$. Thus, near $p$, $\Sigma$
can be represented as the graph of a smooth function $w(x)$ defined on $B_r(p)$ for some $r>0$, which solves (\ref{B-1}).
Since $p_i\in\mathcal{O}$, there exists a neighborhood $\mathcal{O}_{p_i}$ where $w(x)\equiv 0$. As $p_i\to p$, this implies that $w(x)$ vanishes to infinite order at $p$, i.e., all derivatives of $w(x)$ vanish at $p$. Considering the translated function $\tilde{w}(y):=w(y+p)$ in $B_r(0)$, and applying Theorem \ref{Theorem 3.3} again, we conclude that $\tilde{w}(y)\equiv 0$ in some $B_{r_0}(0)$. Hence, $w(x)\equiv 0$ in $B_{r_0}(p)$, which means $p\in\mathcal{O}$. 
An alternative proof follows from the strong unique continuation property for second-order elliptic PDEs (Theorem 1.1 in \cite{[GL-1987]} by N. Garofalo and F.H. Lin).

Since $\mathcal{O}$ is both open and closed in the connected hypersurface $\Sigma$, we have $\Sigma=\mathcal{O}\subseteq \mathbb{R}^n$. Moreover, $\Sigma$ is complete, the Hopf-Rinow theorem implies 
$\Sigma =\mathbb{R}^n$.

\hfill$\Box$

\noindent \textbf{Acknowledgments}.
This work is supported by the National Natural Science Foundation of China (NSFC) (Grant No. 11871445, 2025YFA1017601), the Stable Support Plan for Youth Teams in Basic Research Fields, Chinese Academy of Sciences (CAS) (Grant No. YSBR-001), and the Fundamental Research Funds for the Central Universities.

\end{document}